\documentclass[12pt, leqno]{amsart}
\usepackage[mathscr]{eucal}
\usepackage{color}
\usepackage{relsize}
\usepackage{makecell}
\usepackage{verbatim}
\usepackage{amsmath, amssymb, amsfonts}
\usepackage{a4wide}
\usepackage[normalem]{ulem}

\def\Vec#1{\mbox{\boldmath $#1$}}
\theoremstyle{plain}
\newtheorem{theorem}{Theorem}[section]
\newtheorem{proposition}{Proposition}[section]
\newtheorem{corollary}{Corollary}[section]
\newtheorem{lemma}{Lemma}[section]
\newtheorem{remark}{Remark}[section]

\newtheorem{example}{Example}[section]

\numberwithin{equation}{section}
\title[Magnetic vector fields]{Magnetic unit vector fields}

\author[J.~Inoguchi]{Jun-ichi Inoguchi}
\address[J.~Inoguchi]{Institute of Mathematics, 
University of Tsukuba, 1-1-1 Tennodai Tsukuba, Ibaraki,
305-8571, Japan}
\email{inoguchi@math.tsukuba.ac.jp}
\address[J.~Inoguchi]{{\sl Present address: }
Department of Mathematics,
Hokkaido University,
Sapporo,
060-0810 Japan
}
\email{inoguchi@math.sci.hokudai.ac.jp}
\author[M.~I.~Munteanu]{Marian Ioan Munteanu}
\address[M.~I.~Munteanu]
{Department of Mathematics, 
University  Alexandru Ioan Cuza Iasi, 
Bd. Carol I, n. 11, 
Iasi, 700506, Romania}
\email{marian.ioan.munteanu@gmail.com}

\date{\today}
\dedicatory{Dedicated to professor Kazumi Tsukada on the occasion of his retirement}

\begin{document}

\vspace{18mm}
\setcounter{page}{1}
\thispagestyle{empty}

\begin{abstract}
We show that a unit vector field on an oriented Riemannian manifold is a critical point of the
Landau Hall functional if and only if it is a critical point of the Dirichlet energy functional.
Therefore, we provide a characterization for a unit vector field to be a magnetic map into its unit tangent sphere bundle.
Then, we classify all magnetic left invariant unit vector fields on $3$-dimensional Lie groups.
\end{abstract}

\maketitle

\section{Introduction}
\footnote{{\small accepted to
Revista de la Real Academia de Ciencias Exactas, Físicas y Naturales. 
Serie A. Matem\'aticas}}
Smooth vector fields appear in physics or technology as 
mathematical models of ``fields". For instance, a static magnetic field 
is a smooth vector field $\Vec{B}$ on a region of the 
Euclidean space $3$-space $\mathbb{E}^3$ satisfying 
the \emph{Gauss' law} $\mathrm{div}\>\Vec{B}=0$.
A static magnetic filed $\Vec{B}$ is \emph{uniform} if it
is covariantly constant. Thus parallel vector fields are mathematical models 
of uniform static magnetic field.

Now let us turn our attention to smooth vector fields on 
(oriented) Riemannian manifolds. 
The \emph{bending energy} (biegung) 
$\mathcal{B}$ of a vector field $X$ 
on an oriented Riemannian manifold $(M,g)$ is 
\[
\mathcal{B}(X)=\int_{M}|\nabla{X}|^{2}\,dv_g.
\]
The bending energy of vector fields measures to what extent 
globally defined vector fields fail to be parallel.

Additionally, we may consider 
\emph{Dirichlet energy} 
\[
E(X)=\int_{M}\frac{1}{2}|dX|^{2}\,dv_g
\]
of a vector field $X$. Here we regard $X$ as a smooth map 
from $M$ to its tangent bundle $TM$ (equipped with a natural 
Riemannian metric $g_{_S}$ derived from $g$). Critical points 
of the Dirichlet energy are called \emph{harmonic maps} \cite{EL}.
These two variational problems are closely related each other. 
In fact, when $(M,g)$ is compact and oriented, we have 
\[
2E(X)=\dim M\cdot\mathrm{Vol}(M)+\mathcal{B}(X).
\] 
To investigate these functionals, we should be nervous and careful about 
mapping space and variations. 
In fact, the bending energy is considered
on the space $\mathfrak{X}(M)$ of all smooth vector fields. 
To deduce the Euler-Lagrange equation we compute the 
first variation through variations 
$\{X^{(s)};\>|s|<\varepsilon\}\subset\mathfrak{X}(M)$.

On the other hand, usually Dirichlet energy is 
considered on the space $C^{\infty}(M,TM)$ of all 
smooth maps from $M$ into $TM$. The first variation 
is computed for variations 
$\{
X^{(s)};\>|s|<\varepsilon
\}\subset C^{\infty}(M,TM)$.

Unfortunately, the later setting is quite discouraging.
Actually, Ishihara and Nouhaud showed that the only vector field 
which are harmonic maps are parallel vector fields if $M$ is compact \cite{Ishihara, Nouhaud}.
Even if we restrict the variations in 
$\mathfrak{X}(M)$, the same conclusion holds
(showed by Gil-Medrano~\cite{Gil-Med}).

One of the appropriate mapping space for Dirichlet or bending energy 
is the space $C^{\infty}(M,UM)$ of all smooth maps 
form $M$ into its unit tangent sphere bundle $UM$ or 
the space $\mathfrak{X}_{1}(M)$ of all smooth unit vector fields.
In fact, a unit vector field $X$ is a critical point 
of $\mathcal{B}$ through (compactly supported) 
variations in $\mathfrak{X}_{1}(M)$
if and  only if 
\begin{equation}\label{eq:HUVT}
\overline{\Delta}_gX=|\nabla X|^2X,
\end{equation}
where 
$\overline{\Delta}_g$ is the \emph{rough Laplacian} (see \eqref{eq:roughL}).
Unit vector fields satisfying \eqref{eq:HUVT} are referred as to 
\emph{harmonic unit vector fields} \cite{DP}.

Next, $X$ is a harmonic map into $UM$, 
that is, critical point 
of $E$ through (compactly supported)  
variations in $C^{\infty}(M,UM)$
if and  only if $X$ satisfies \eqref{eq:HUVT} and, in addition,
\begin{equation*}
\label{eq:HMUVT}
\mathrm{tr}_{g}R(\nabla_\cdot X,X)\cdot=0,
\end{equation*}
where $R$ is the Riemannian curvature. 
Harmonic unit vector fields have been studied 
extensively. See a monograph by Dragomir and Perrone \cite{DP}. 
For example, Gonz{\'a}lez-D{\'a}vila and Vanhecke \cite{GV2002UMI}
classified left invariant unit harmonic vector fields on
$3$-dimensional Lie groups. Accordingly, in the non-unimodular case,
the left-invariant harmonic unit vector fields determining 
harmonic maps are vary rare. In \cite{GV}, the same authors study the stability and 
instability of harmonic and minimal unit vector fields on three-dimensional compact manifolds, 
in particular on compact quotients of unimodular Lie groups.

Furthermore, every unit vector field is an immersion of $M$ into $UM$. 
Thus we may develop submanifold geometry of unit vector fields in $UM$.
For example, the study of minimal vector fields was initiated by Gluck and Ziller in \cite{GZ}.
They studied \emph{optimal unit vector fields} on the
unit $3$-sphere $\mathbb{S}^3$. 
In \cite{HY}, Han and Yim proved that Hopf vector fields on odd dimensional spheres ${\mathbb{S}}^{2m+1}$
are harmonic maps into $U\mathbb{S}^{2m+1}$. In the same paper \cite{HY}, a converse is proved for $m=1$, 
that is a unit vector field on $\mathbb{S}^3$ is a harmonic map into $U\mathbb{S}^3$ if and only if it is a
Hopf vector field on $\mathbb{S}^3$. In fact, Hopf vector fields on $\mathbb{S}^{2m+1}$ are precisely the unit
Killing vector fields \cite{W96}. Later on,
Tsukada and Vanhecke studied minimal unit vector fields on Lie groups.
In particular, they classified left invariant minimal unit vector fields on
$3$-dimensional Lie groups \cite{TV}. 
It should be remarked that model spaces of Thurston geometries are realized as $3$-dimensional 
Lie groups with left invariant metric except the product space $\mathbb{S}^2\times
\mathbb{R}$. These observations suggest us to introduce a new variational problem
to look for ``optimal" unit vector fields.

Motivated by static magnetic theory, we have 
started our investigation on \emph{magnetized harmonic vector fields}
(magnetic vector fields, in short) 
\cite{IM2,IM3}. A vector field $X$ of $(M,g)$ is said to be 
\emph{magnetic} if it is a critical point of the 
Landau-Hall functional, which is a
magnetization of Dirichlet energy (see \S \ref{sec:2.3} for the 
precise definition).

In this paper we give a  detailed investigation on 
magnetized unit vector fields on Riemannian manifolds.
In particular, we classify magnetic left invariant unit vector fields on $3$-dimensional Lie groups.

\section{Magnetic maps}  
\subsection{Magnetic fields}
Let $(N,h)$ be an $n$-dimensional Riemannian manifold.
A \emph{magnetic field} is a closed 2-form $F$ on $N$
and the \emph{Lorentz force} of a magnetic field $F$ on $(N,h)$ is 
the unique $(1,1)$-tensor field $L$ given by
\begin{equation*}
\label{Lorentzforce}
g(LX, Y) = F(X, Y), \quad \forall X,Y\in {\mathfrak{X}}(N).
\end{equation*}
\subsection{Energy and tension} 
Let $f:(M,g)\longrightarrow (N,h)$ be a smooth maps between two Riemannian manifolds
$(M,g)$ of dimension $m$ and $(N,h)$ of dimension $n$. Assume that $M$ is 
\emph{oriented}. The \emph{Dirichlet energy} of a smooth map $f:M\to N$ 
over a closed region $\mathsf{D}\subset M$ is  
defined by
\[
E(f;\mathsf{D})=\int_{\mathsf{D}}\frac{1}{2}|df|^2\,dv_g.
\]
When $M$ is compact, we use an abbreviation $E(f)=E(f;M)$. 

The \emph{second fundamental form} $\nabla df$ of $f$ is defined by
\[
(\nabla df)(X,Y)=\nabla^{f}_{X}f_{*}Y
-f_{*}(\nabla_XY),
\ \  X,Y\in\mathfrak{X}(M).
\]
We adopt the notations $\nabla$ and $\nabla^N$ for the Levi-Civita connections of $M$ and $N$,
respectively and $\nabla^f$ for the connection in the pull-back bundle $f^{*}TN$.
The \emph{tension field} $\tau(f)$ is a vector field along $f$, that is a section of $f^{*}TN$ defined by
\[
\tau(f)=\mathrm{tr}_{g}(\nabla df)=\sum_{i=1}^{m}(\nabla df)(e_i,e_i),
\]
where $\{e_1,e_2,\cdots, e_m\}$ is a local orthonormal frame field of $M$.

A smooth map $f:M\to N$ is said to be a 
\emph{harmonic map} if it is a critical point of the Dirichlet energy over all
compactly supported regions in $M$ \cite{EL}.
The Euler-Lagrange equation of the variational problem of the Dirichlet energy 
is 
\[
\tau(f)=0.
\]
In case when $f$ is an isometric immersion, then 
$\tau(f)=mH$, where $H$ is the mean curvature vector field.
Thus an isometric immersion is a harmonic map if and only if it is a 
minimal immersion.

\subsection{Magnetic maps}
\label{sec:2.3}
Let $(M,g)$ and $(N,h)$ be Riemannian manifolds as before. 
Take a magnetic field $F$ on $N$ and a vector field $\Xi$ on $M$.
A smooth map $f:(M,g)\to (N,h,F)$ is said to be a
\emph{magnetic map} (or \emph{magnetized harmonic map}) of 
\emph{charge} $q$ with the directional vector field $\Xi$, if 
it satisfies 
\[
\tau(f)=q\,Lf_{*}\Xi ~,
\]
for some constant $q$. 
The prescribed vector field $\Xi$ controls 
the direction of the movement of the image of the magnetic map $f$ in $N$ under the influence of magnetic field 
$F$. See e.g. \cite{IM1}.

Under the assumption $\Xi$ is \emph{divergence free} and $F$ is 
\emph{exact}, 
magnetic maps are characterized as critical points
of the \emph{Landau-Hall functional}:
\[
\mathrm{LH}(f;\mathsf{D})=E(f;\mathsf{D})+q\int_{\mathsf{D}}A(f_{*}\Xi)\,dv_{g}
\]
under all compactly supported variations \cite{IM1}. Here $A$ is a potential $1$-form of $F$ 
(called the \emph{magnetic potential}).

\section{Magnetic vector fields}
\subsection{The tangent bundles}
As is well known, the tangent 
bundle of a Riemannian manifold admits an almost 
K{\"a}hler structure, naturally induced from 
the Riemannian structure of the 
base manifold. In particular, the canonical 
symplectic form of the tangent bundle, 
is a magnetic field.
We briefly present a collection of basic materials 
about the geometry of tangent bundles for the study of vector fields \cite{IM4}. 

Let $TM$ be the tangent bundle of a Riemannian manifold 
$(M,g)$ with the projection $\pi:TM \to M$. 
The \emph{vertical distribution} $\mathcal{V}$ is defined by 
\[
\mathcal{V}_{(p;u)}=\mathrm{Ker}(\pi_{*(p;u)}),\quad   u=(p;u)\in T_{p}M\subset TM.
\]
The Levi-Civita connection $\nabla$ induces a complementary 
distribution $\mathcal{H}$, that is a distribution satisfying  
$T(TM)=\mathcal{H}\oplus \mathcal{V}$. The distribution 
$\mathcal{H}$ is called a \emph{horizontal distribution} determined by 
$\nabla$. In addition there exists a linear isomorphism 
$\mathsf{h}=\mathsf{h}_{(p;u)}:T_{p}M\to \mathcal{H}_{(p;u)}$. 
The isomorphism $\mathsf{h}$ is called the 
\emph{horizontal lift operation}. There exists also 
a linear isomorphism 
$\mathsf{v}=\mathsf{v}_{(p;u)}:T_{p}M\to \mathcal{V}_{(p;u)}$. 
The isomorphism $\mathsf{v}$ is called the 
\emph{vertical lift operation}. These operations are naturally 
extended to vector fields.

The Riemannian metric $g$ of $(M,g)$ induces an almost
K{\"a}hler structure $(g_{_S},J)$ on the tangent bundle $TM$ of $M$:
\[
g_{_S}(X^\mathsf{h},Y^\mathsf{h})=
g_{_S}(X^\mathsf{v},Y^\mathsf{v})=g(X,Y)\circ \pi,
 \ \ 
g_{_S}(X^\mathsf{h},Y^\mathsf{v})=0,
\]
\[
JX^{\mathsf{h}}=X^{\mathsf{v}},\ \ 
JX^{\mathsf{v}}=-X^{\mathsf{h}},
\]
for all $X$, $Y\in\mathfrak{X}(M)$. 
The metric $g_{_S}$ is the so-called \emph{Sasaki lift metric} 
of $TM$. 
The fundamental $2$-form $F$ of 
$(TM,g_{_{S}},J)$ 
defined by 
$F(\cdot,\cdot)=g_{_{S}}(J\cdot,\cdot)$ is a symplectic form on $TM$ and called 
the \emph{canonical $2$-form} or \emph{canonical symplectic form} 
of $TM$. Since $F$ is closed, we regard $F$ as a 
canonical magnetic field on $TM$. It should be 
remarked that the canonical magnetic field $F$ is exact. Indeed, 
$F$ is represented as $F=-d\omega$. Here the $1$-form $\omega$ 
(called the \emph{canonical $1$-form}) is 
defined by
\[
\omega_{(p;u)}(X^{\mathsf h})=g(u,X_p),
\quad 
\omega_{(p;u)}(X^{\mathsf v})=0.
\]
Here we use the {\sl{Det}} convention.
 Hereafter, we equip $TM$ with the magnetic field $F=-d\omega$.

\subsection{Tension fields}
Take a vector field $X$ on $M$. Regarding $X$ as a smooth map 
from $(M,g)$ into $(TM,g_{_S},F)$, then the tension 
field $\tau(X)$ of $X\in C^{\infty}(M,TM)$ is given by \cite{Ishihara}:
\[
\tau(X)=-\left\{(
\mathrm{tr}_{g}R(\nabla_\cdot X,X)\cdot\,)^{\mathsf{h}}
+(\overline{\Delta}_{g}X)^{\mathsf{v}}
\right\}\circ X,
\] 
where $\mathsf{h}$ and $\mathsf{v}$ are horizontal lift and 
vertical lift from $M$ to $TM$, respectively.
The operator $\overline{\Delta}_{g}$ acting on $\mathfrak{X}(M)$ is called the 
\emph{rough Laplacian} and it is defined by
\begin{equation}
\label{eq:roughL}
\overline{\Delta}_{g}=-\sum_{i=1}^{m}
\left(\nabla_{e_i}\nabla_{e_i}-\nabla_{\nabla_{e_i}e_i}
\right),
\end{equation}
where $\{e_1,e_2,\cdots, e_m\}$ is a local orthonormal frame field of $M$ as before.

\begin{theorem}[\cite{Gil-Med}]
A vector field $X:M\to TM$ is a critical point of the Dirichlet energy 
with respect to all compactly supported variations in $\mathfrak{X}(M)$ 
if and only if 
$\overline{\Delta}_{g}X=0$.
\end{theorem}
Note that vector fields which are critical points of the Dirichlet energy 
with respect to all compactly supported variations in $\mathfrak{X}(M)$ 
are called \emph{harmonic vector field} in some literature; see e.g. \cite{DP}
and references therein.

\subsection{}
Let $\varDelta_g$ be the Laplace-Beltrami operator of $(M,g)$ acting on 
$1$-forms. Via  the musical isomorphisms between $TM$ and $T^{*}M$, 
we graft $\varDelta_g$ for vector fields (denoted by the same symbol $\varDelta_g$). 
A vector field $X$ is said to be $\varDelta_g$-\emph{harmonic} if $\varDelta_g X=0$. 

The following \emph{Weitzenb{\"o}ck formula} measures the difference between
 $\varDelta_{g}$ and $\overline{\Delta}_g$
\[
\varDelta_{g}X=\overline{\Delta}_gX+SX,
\]
where $S$ is the \emph{Ricci operator} of $M$.

\begin{corollary}
A vector field $X:M\to TM$ is a critical point of the Dirichlet energy 
with respect to all compactly supported variations in $\mathfrak{X}(M)$ 
if and only if it satisfies $\varDelta_{g}X=SX$.
\end{corollary}

For a vector field, the condition of being harmonic map is much stronger than that of being 
harmonic vector field.

\begin{theorem}[\cite{Gil-Med}]
A vector field $X:(M,g)\to (TM,g_{_S})$ is a harmonic map if and only if 
\[
\mathrm{tr}_{g}R(\nabla_\cdot X,X)\cdot=0,
\quad
\overline{\Delta}_{g}X=0.
\]
\end{theorem}
See also \cite[Proposition 2.12]{DP}.

As we have also mentioned in the Introduction, 
harmonicity of vector field is a too strong restriction for vector fields.
Actually, Gil-Medrano \cite{Gil-Med} showed the following fact.
\begin{theorem}
Every harmonic vector field on a compact oriented Riemannian manifold $M$ 
is a parallel vector field.
\end{theorem}

\begin{corollary}[Ishihara, Nouhaud]
Every vector field on a compact oriented Riemannian manifold $M$, 
which is a harmonic map into $TM$, is parallel.
\end{corollary}

\subsection{Magnetic vector fields}
Since the harmonicity is a too strong restriction for vector fields, 
we weaken this property to ``magnetic".
Namely, we consider vector fields which are magnetic maps into $TM$. 
A vector field $X$ on $M$ is a {\it magnetic map} with charge $q$ if $X$ satisfies 
$\tau(X)=qJ(X_{*}X)$.

In our previous paper \cite{IM2} we obtained the 
following fundamental formula:
\begin{equation}
\label{eq:X*Y}
X_{*p}Y_{p}=Y^{\mathsf{h}}_{X_p}+(\nabla_{Y}X)^{\mathsf{v}}_{X_p}
\end{equation}
for any vector fields $X$ and $Y$ on $M$.

From this formula, we have
\[
J(X_{*}X)=X^{\mathsf{v}}_{X}-(\nabla_{X}X)^{\mathsf{h}}_X.
\]
Hence the magnetic map equation $\tau(X)=qJ(X_{*}X)$
is the system \cite{IM2}:
\[
\mathrm{tr}_{g}R(\nabla X,X)=q\,\nabla_{X}X,
\quad 
\overline{\Delta}_{g}X=-q X.
\]

In \cite{IM2}, some examples of vector fields satisfying this system are exhibited.

\subsection{LH-critical vector fields} 
Next, we consider again a vector field preserving variation $\{X^{(s)}\}$ through 
a vector field $X$. 
More precisely $\{X^{(s)}\}$ is a smooth map from the 
product manifold $(-\varepsilon,\varepsilon)\times M$ to $TM$ satisfying 
\[
(-\varepsilon,\varepsilon)\times M\ni(s,p)\longmapsto X^{(s)}_p\in T_{p}M,
\quad X^{(0)}_p=X_p.
\]
Here $\varepsilon$ is a sufficiently small positive number. 
The variation $\{X^{(s)}\}$ is interpreted as a smooth map
\[
\chi:(-\varepsilon,\varepsilon)\times M\to TM;
\>\>
\chi(s,p)=X^{(s)}_p.
\]
It should be remarked that for every $s\in (-\varepsilon,\varepsilon)$, 
the correspondence 
$p\longmapsto \chi(s,p)=X^{(s)}_p$ is a vector field on $M$. 
In other words, $X^{(s)}$ is understood as a vector field 
$X^{(s)}:p\longmapsto X^{(s)}_p$.
Note that $X^{(0)}(p)=X_p$ for any $p\in M$. 
Next, since every $X^{(s)}$ is a vector field, 
\[
\pi (X^{(s)}(p))=p
\]
holds for any $s$ and $p\in M$. 

The \emph{variational vector field} $\mathscr{V}$ 
of the variation $\{X^{(s)}\}$ is defined by
\[
\mathscr{V}_{p}=\chi_{*(0,p)}\,
\frac{\partial}{\partial s}\biggr\vert_{(0,p)}.
\]
By definition $\mathscr{V}$ is a section of the 
pull-back bundle $X^{*}T(TM)$.

Since $X$ is an immersion into $TM$, 
one can introduce the fiber metric 
on the pull-back bundle $X^{*}T(TM)$ by 
$X^{*}g_{_S}$.
Then the first variation for the Dirichlet energy is 
given by
\[
\frac{d}{ds}\biggr\vert_{s=0}
E(X^{(s)};{\mathsf{D}})=
-\int_{\mathsf{D}}(X^{*}g_{_S})(\mathscr{V},\tau(X))\,
dv_{g}.
\]
On the other hand, the \emph{tangential vector field} $V$ of 
$\{X^{(s)}\}$ is defined by (see \cite[p.~57]{DP}):
\[
V_p=\frac{d}{ds}\biggr\vert_{s=0}X^{(s)}(p)
=\lim_{s\to 0}\frac{1}{s}\left(
X^{(s)}(p)-X_p\right)\in T_{p}M.
\]
The tangential vector field and the variational vector field are related by
\cite[p.~58]{DP}
\[
\mathscr{V}=V^{\mathsf v}\circ X.
\]
Hence, the first variation for the Dirichlet energy may be expressed as
\begin{align*}
\frac{d}{ds}\biggr\vert_{s=0}
E(X^{(s)};{\mathsf{D}})=&
-\int_{\mathsf{D}}(X^{*}g_{S})(\mathscr{V},\tau(X))\,
dv_{g}
\\
=&
\int_{\mathsf{D}}g_{S}(V^{\mathsf{v}}\circ X,
\left\{(
\mathrm{tr}_{g}R(\nabla_\cdot X,X)\cdot\,)^{\mathsf{h}}
+(\overline{\Delta}_{g}X)^{\mathsf{v}}
\right\}\circ X)\,
dv_{g}
\\
=&
\int_{\mathsf{D}}g(V,\overline{\Delta}_{g}X)
\,dv_g.
\end{align*}
Now let us compute the first variation for the 
magnetic term 
of the Landau-Hall functional, 
namely
\[
\frac{d}{ds}\biggr\vert_{s=0}\int_{\mathsf{D}}\omega_{\chi(s,p)}((X^{(s)})_{*p}X_{p})\,dv_g.
\]
Since $X^{(s)}$ is a vector field, by using 
\eqref{eq:X*Y} we get
\[
\left(X^{(s)}\right)_{*p}X_{p}=X^{\mathsf h}_{\chi(s,p)}
+\left(\nabla_{X}X^{(s)}\right)^{\mathsf v}_{\chi(s,p)}.
\]
From this formula we have
\[
\omega_{{_\chi(s,p)}}\left(
\left(
X^{(s)}
\right)_{*p}X_p
\right)
=
\omega_{_{\chi(s,p)}}\left(
X^{\mathsf h}_{\chi(s,p)}
\right)
=g_{p}(X^{(s)}_p,X_p).
\]
Hence 
\[
\frac{d}{ds}\biggr\vert_{s=0}
\omega_{_{\chi(s,p)}}((X^{(s)})_{*p}X_p)
=g_{_p}(V,X).
\]
The first variation of the magnetic term 
$\displaystyle\int_{\mathsf{D}}\omega_{_{\chi(s,p)}}((X^{(s)})_*X)\,dv_{g}$ is 
computed as
\begin{align*}
\frac{d}{ds}\biggr\vert_{s=0}
\int_{\mathsf{D}}
\omega((X^{(s)})_*X)\,dv_{g}=&
\int_{\mathsf{D}}
\frac{d}{ds}
\omega\left(\left(X^{(s)}\right)_*X\right)\,
\biggr\vert_{s=0}\,dv_{g}
\\
=& \int_{\mathsf{D}}
g(V,X) \,dv_{g}.
\end{align*}
Thus we arrive at the following result:
\begin{align*}
\frac{d}{ds}\biggr\vert_{s=0}
\mathrm{LH}(X^{(s)};{\mathsf{D}})=&
\frac{d}{ds}\biggr\vert_{s=0}
\left(E(X^{(s)};{\mathsf{D}})+q 
\int_{\mathsf{D}}\omega
\left(\left(X^{(s)}\right)_{*}X
\right)\,
dv_g\right)
\\
=&
\int_{\mathsf{D}}g(V,\overline{\Delta}_{g}X+qX)
\,dv_g.
\end{align*}
Thus we obtain the Euler-Lagrange equation for Landau-Hall functional 
on $\mathfrak{X}(M)$.
\begin{theorem}
A vector field $X$ on an oriented Riemannian manifold $(M,g)$ 
is a critical point of the Landau-Hall functional under 
compact supported variations in $\mathfrak{X}(M)$ 
if and only if 
\[
\overline{\Delta}_{g}X=-qX.
\]
\end{theorem}

\section{Magnetic unit vector fields}
\subsection{}
Let us denote by $UM$ the unit tangent sphere bundle 
of $(M,g)$. Then $UM$ is a hypersurface of $TM$ with 
unit normal vector field $\Vec{U}$. 
Here $\Vec{U}$ is the so-called 
canonical vertical vector field of $TM$.
The key formula we use for $\Vec{U}$ is 
\[
\Vec{U}_{(p;w)}=w^{\mathsf v}_{w}
\]
for any $w\in T_{p}M$.

The almost K{\"a}hler structure $(g_{_S},J)$ induces an 
almost contact metric structure $(\phi,\xi,\eta,g_{_S})$ on $UM$ by
\[
J V=\phi V+\eta(V)\Vec{U},\ \ \xi=-J\Vec{U}.
\]
The $1$-form $\eta$ is a contact form on $UM$. Moreover $\phi$ is skew-adjoint 
with respect to the induced metric and hence $F_U(\cdot,\cdot)=g_{_S}(\phi \cdot,\cdot)$ is a magnetic field 
with Lorentz force $\phi$. Moreover $F_U$ is exact; in fact, 
$F_U=-d\eta$. 

For any vector field $Y$ on $M$ and a unit tangent 
vector $u=(p;u)\in T_{p}M$, its 
horizontal lift $Y^{\mathsf{h}}_{u}$ is tangent to $UM$, 
\textit{i.e.}, $Y^{\mathsf{h}}_{u}\in T_{u}(UM)$.
On the other hand, $Y^{\mathsf{v}}_{u}$ is not always tangent to $UM$.

The \emph{tangential lift} $Y^{\mathsf{v}}_{u}$ is defined by
\[
Y^{\mathsf{t}}_{u}=Y^{\mathsf{v}}_{u}-g(u,Y_p)\Vec{U}_{u}\in 
T_u(UM).
\]

The contact form $\eta$ is given explicitly by
\[
\eta_{(p;u)}(X^{\mathsf h}_{u})=g_{p}(X_p,u),
\ \ 
\eta_{(p;u)}(X^{\mathsf t}_{u})=0.
\]
\subsection{Tension field} 
Let us take now a \emph{unit vector field} $X$ on $M$ and regard it as a 
smooth map $X:M\to UM$.
Then its tension field $\tau_{_1}(X):=\tau(X;UM)$ is 
computed as (see \cite[p.~59]{DP}, \cite{HY}): 
\[
\tau_{_1}(X)=-\left\{(
\mathrm{tr}_{g}R(\nabla_\cdot X,X)\cdot\,)^{\mathsf{h}}
+(\overline{\Delta}_{g}X)^{\mathsf{t}}
\right\}\circ X.
\]
Next, we compute $\phi(X_{*}X)$.
We know that
\[
J(X_{*}X)_{X_p}=\phi(X_{*}X)_{X_p}+\eta(X_{*}X)_{X_p}X^{\mathsf{v}}_{X_p}.
\]
On the other hand, 
\[
J(X_{*}X)_{X}=X^{\mathsf{v}}_{X}-(\nabla_{X}X)^{\mathsf{h}}_{X}
=X^{\mathsf{t}}_{X}+g(X,X)\Vec{U}_{X}-(\nabla_{X}X)^{\mathsf{h}}_{X}.
\]
It follows that
\[
\phi(X_{*}X)_{X}=X^{\mathsf{t}}_{X}
-(\nabla_{X}X)^{\mathsf{h}}_{X},\  \ \eta(X_{*}X)_{X}=g(X,X)=1.
\]
Hence the magnetic map equation $\tau_{_1}(X)=q\phi(X_{*}X)$
is the system
\[
\mathrm{tr}_{g}(R(\nabla_\cdot X,X)\cdot\,)^{\mathsf{h}}_{X}=q(\nabla_{X}X)^{\mathsf{h}}_{X},
\ \ 
(\overline{\Delta}_{g}X)^{\mathsf{t}}_{X}=-qX^{\mathsf{t}}_{X}.
\]
The first equation is equivalent to
\[
\mathrm{tr}_{g}R(\nabla_\cdot X,X)\cdot=-q\nabla_{X}X.
\] 
The second equation is rewritten as
\[
(\overline{\Delta}_{g}X)^{\mathsf{v}}_{X}-g(X,\overline{\Delta}_{g}X)\Vec{U}_X
=-q(X^{\mathsf{v}}_{X}-g(X,X)\Vec{U}_X).
\]
This equation is equivalent to
\[
\overline{\Delta}_{g}X-g(X,\overline{\Delta}_{g}X)X=-q(X-g(X,X)X)=0.
\]
Thus we obtain
\[
\overline{\Delta}_{g}X=g(X,\overline{\Delta}_{g}X)X
=|\nabla X|^2\,X.
\]
\begin{theorem}
\label{thm4.1}
A unit vector field $X$ on $(M,g)$ is a magnetic map into $UM$ if and only if 
\begin{equation}\label{Meq}
\mathrm{tr}_{g}R(\nabla_\cdot X,X)\cdot=q\nabla_{X}X,
\ \ 
\overline{\Delta}_{g}X=|\nabla X|^2\,X.
\end{equation}
\end{theorem}

\subsection{}
Next, we consider a 
unit vector field preserving variation $\{X^{(s)}\}$ through 
a unit vector field $X$. About Dirichlet energy the following result is 
known.

\begin{theorem}[\cite{HY,CMWood, Wiegmink}]
A unit vector field $X$ on an oriented Riemannian manifold $(M,g)$ 
is a critical point of the Dirichlet energy with respect to 
compactly supported variations in $\mathfrak{X}_{1}(M)$ if and only if 
\[
\overline{\Delta}_{g}X=|\nabla X|^2\,X.
\]
\end{theorem}
A unit vector field $X$ satisfying $\overline{\Delta}_{g}X=|\nabla X|^2\,X$ is 
referred as to a {\emph{harmonic unit vector field}} \cite{DP}.

In order to compute the first variation for the Landau-Hall functional, let us
make some remarks. Using \eqref{eq:X*Y} we write
$X^{(s)}_*X^{(0)}=(X^{(0)})^\mathsf{h}_{X^{(s)}}+(\nabla_{X^{(0)}}X^{(s)})^\mathsf{v}_{X^{(s)}}$.
But we have also
$g_{_S}\left(\Vec{U},(\nabla_{X^{(0)}}X^{(s)})^\mathsf{v}_{X^{(s)}}\right)=0$, which means that 
$(\nabla_{X^{(0)}}X^{(s)})^\mathsf{v}_{X^{(s)}}=(\nabla_{X^{(0)}}X^{(s)})^\mathsf{t}_{X^{(s)}}$.
As a consequence,
$\eta_{_{X^{(s)}}}(X^{(s)}_{*}X^{(0)})=\eta_{_{X^{(s)}}}\big((X^{(0)})^\mathsf{h}_{X^{(s)}}\big)=g(X^{(s)},X)$.
Thus the variation of the magnetic term is given by
\begin{equation}
\label{eq:first-eta}
\frac{d}{ds}\biggr\vert_{s=0}\int_{\mathsf{D}}\eta({X^{(s)}}_{*}X^{(0)})\,dv_g=\int_{\mathsf{D}}g(V,X)dv_g.
\end{equation}
Here $V$ is the vector field on $D$ previously defined by
$V(p)=\lim\limits_{s\to0}\frac{1}{s}\big(X^{(s)}(p)-X(p)\big)$.

We will prove the following result.
\begin{theorem} 
\label{Th:4.3}
A unit vector field $X$ on an oriented Riemannian manifold $(M,g)$ 
is a critical point of the Landau-Hall functional under 
compact support variations in $\mathfrak{X}_{1}(M)$ 
if and only if it is a critical point 
of the Dirichlet energy 
under compact support variations in $\mathfrak{X}_{1}(M)$.
\end{theorem}
Before giving a proof, it seems to be better to 
add some helpful remarks.
\begin{lemma}
\label{Lem-4.1}
Let $X\in\mathfrak{X}_{1}(M)$ be a unit vector field 
and $X^{(s)}$ be a variation of $X$ through unit vector fields. 
Then, the vector field $V$ obtained from the variation
$X^{(s)}$ of $X$ is orthogonal to $X$.
Conversely, let $V$ be a vector field orthogonal to $X$. Then there exits a 
variation $X^{(s)}$ of $X$ through unit vector fields whose variational vector field leads to $V$.
\end{lemma}
\begin{proof}
Since $X^{(s)}$ is a variation of $X$ through unit vector fields, that is 
$g(X^{(s)},X^{(s)})=1$, we immediately obtain
$0=\frac{d}{ds}\biggr\vert_{s=0}g(X^{(s)},X^{(s)})=2g(V,X)$.
See also \cite[p.~63]{DP}.
Hence the conclusion.

To prove the converse, we set 
$W^{(s)}=X+sV$ and $X^{(s)}=\frac{1}{f(s)}W^{(s)}$, where $f(s)=|W^{(s)}|$.

Obviously, $f(0)=1$ and  
$$
f(s)^2=|W^{(s)}|^2=1+s^2|V|^2.
$$
This implies $f'(0)=0$. Consequently, the variational vector field of $W^{(s)}$ is
$$
\frac{d}{ds}\biggr\vert_{s=0}X^{(s)}=\Big(-\frac{f'(s)}{f(s)^2}W^{(s)}+\frac{1}{f(s)}\frac{d}{ds}W^{(s)}\Big)_{s=0}=V.
$$
We conclude that $X^{(s)}$ is a required variation. (See e.g. \cite[p.~65]{DP}.) 
\end{proof}
\begin{lemma}
Let $X\in\mathfrak{X}_{1}(M)$ be a unit vector field 
and $X^{(s)}$ be a variation of $X$ through unit vector fields whose variational vector field is $V$.
Let 
$$
{\mathrm{LH}}(s):={\mathrm{LH}}(X^{(s)};\mathsf{D})=E(X^{(s)};\mathsf{D})+
  \displaystyle q\int_\mathsf{D}\eta_{_{X^{(s)}}}(X^{(s)}_{*}X^{(0)})dv_g
$$
be the Landau-Hall functional under compact support variations in $\mathfrak{X}_{1}(M)$.  
Then
$$
{\mathrm{LH}}'(0)=\int_{\mathsf{D}}g(V,\overline{\Delta}_gX+qX)dv_g.
$$
\end{lemma}
\begin{proof}
The first variation of the Dirichlet energy is given by
\cite{Wiegmink, CMWood} (see also \cite[p.~65]{DP}):
\[
\frac{d}{ds}\biggr\vert_{s=0}
E(X^{(s)};\mathsf{D})
=\int_{\mathsf{D}}g(V,\overline{\Delta}_gX)dv_g.
\]
Combining with \eqref{eq:first-eta} we obtain that the first variation formula of the Landau-Hall functional 
is given by
\[
\frac{d}{ds}\biggr\vert_{s=0}
\mathrm{LH}(X^{(s)};\mathsf{D})
=
\int_{\mathsf{D}}
g(V,\overline{\Delta}_gX+qX)dv_g.
\]
\end{proof}
\begin{lemma} 
Let $X$ be a unit vector field and $V$ a vector field 
orthogonal to $X$. Assume that $X$ is a critical 
point of the Landau-Hall functional under 
unit vector field preserving variations. 
Then we have
\[
\int_{\mathsf D}
g(V,\overline{\Delta}_gX+qX)dv_g=0.
\]
\end{lemma}
\begin{proof}
It is enough to choose $X^{(s)}=W^{(s)}/|W^{(s)}|$ (see Lemma~\ref{Lem-4.1}) in the 
variational formula obtained in the previous Lemma.
\end{proof}
\begin{proof}[Proof of the  {\bf Theorem}~{\rm \ref{Th:4.3}}]
Let $X$ be a magnetic unit vector field with charge $q$. Then,
from the first variational formula, we deduce that 
\[
\int_\mathsf{D}g(V,\overline{\Delta}_gX+qX)dv_g=0
\]
for all vector fields $V$ orthogonal to $X$.

Let us decompose the tangent space $T_{p}M$ as 
$T_{p}M=\mathbb{R}X_{p}\oplus
(\mathbb{R}X_{p})^{\perp}$.
Accordingly, we consider 
\[
\overline{\Delta}_{g}X+qX=\lambda X+V,
\] 
for some smooth function $\lambda$. This equation implies
$g(\overline{\Delta}_{g}X+qX,V)=|V|^2$.
Thus we get
\[
\int_{\mathsf{D}} |V|^2dv_g=
\int_{\mathsf{D}}g(\overline{\Delta}_{g}X+qX,V)dv_g=0.
\]
Hence $V=0$. Thus $\overline{\Delta}_{g}X+qX=\lambda X$. 
From this we have
\[
\lambda=g(\overline{\Delta}_g X+qX,X)=|\nabla X|^2+q,
\]
which implies
\[
\overline{\Delta}_{g}X=|\nabla X|^2~X.
\]

Conversely, assume that $X$ satisfies $\overline{\Delta}_{g}X=|\nabla X|^2X$.
Take a variation $X^{(s)}$ with variational vector field 
$V$, which is orthogonal to $X$. See Lemma~\ref{Lem-4.1}.
We compute
\[
\frac{d}{ds}\biggr\vert_{s=0}\mathrm{LH}(X^{(s)})
=\int_{\mathsf D}g(\overline{\Delta}_g X+qX,V)\,dv_g
=\int_{\mathsf D}g((|\nabla X|^2+q) X,V)\,dv_g=0.
\]
Hence $X$ is a critical point of the Landau-Hall
functional under unit vector field preserving 
variations. 
\end{proof}

It should be remarked that the Euler-Lagrange equation 
of this variational problem is nothing but the 
second equation of magnetic map equations (4.1) for 
unit vector fields. Therefore, we state the following.
\begin{corollary} A unit vector field 
$X$ is a magnetic map into $UM$ with charge $q$ if and only if it 
is a critical point of the Landau-Hall functional under 
compact support variations in $\mathfrak{X}_{1}(M)$ and in addition satisfies 
$\mathrm{tr}_{g}R(\nabla_\cdot X,X)\cdot=q\nabla_{X}X$.
\end{corollary}

Because of the Theorem~\ref{Th:4.3} we give the following definition:
A {\emph{magnetic unit vector field}} on $(M,g)$ is a unit vector field on $M$ such that it is a magnetic
map from $M$ to $U(M)$.

In the next two sections we look for explicit examples 
of unit vector fields satisfying \eqref{Meq}.

\section{Unimodular Lie groups}
\subsection{Unimodular basis}
Let $G$ be a $3$-dimensional unimodular Lie group with a left
invariant metric $g=\langle\cdot,\cdot\rangle$. Then there exists an
orthonormal basis $\{ e_1, e_2, e_3 \}$ of the Lie algebra $\mathfrak{g}$ such that 
\begin{equation}
\label{eq:unimodularbasis}
\ [e_1,e_2]=c_{3}e_{3},\quad  [e_2,e_3]=c_{1}e_{1},\quad
\ [e_3,e_1]=c_{2}e_{2}, \qquad c_{i}\in \mathbb{R}.
\end{equation}

Three-dimensional unimodular Lie groups are classified by Milnor as
follows (signature of structure constants are 
given up to numeration) \cite{Milnor}:
\begin{center}
\begin{tabular}{|c|c|c|}
\hline
   Signature
 of $(c_1,c_2,c_3)$ 
& Simply connected Lie group
& Property \\
  \hline
  $(+,+,+)$
& $\mathrm{SU}(2)$
& compact and simple \\
$(+,+,-)$ & $\widetilde{\mathrm{SL}}_{2}\mathbb{R}$ & non-compact
and simple
\\
$(+,+,0)$ & $\widetilde{\mathrm E}(2)$ & solvable
\\
$(+,-,0)$ & $\mathrm{E}(1,1)$ & solvable
\\
$(+,0,0)$ & Heisenberg group  & nilpotent
\\
$(0,0,0)$ & $(\mathbb{R}^3,+)$ & Abelian
\\
\hline
\end{tabular}
\end{center}%
Denote by the same letters also the left translated vector fields determined by $\{e_1, e_2, e_3\}$.

To describe the Levi-Civita connection $\nabla$ of $G$, we introduce 
the following constants:
\[
\mu_{i}=\frac{1}{2}(c_{1}+c_{2}+c_{3})-c_{i}.
\]
\begin{proposition}
The Levi-Civita connection is given by
\[
\begin{array}{ccc}
\nabla_{e_1}e_{1}=0, & \nabla_{e_1}e_{2}=\mu_{1}e_{3}, & \nabla_{e_1}e_{3}=-\mu_{1}e_{2}\\
\nabla_{e_2}e_{1}=-\mu_{2}e_{3}, & \nabla_{e_2}e_{2}=0, & \nabla_{e_2}e_{3}=\mu_{2}e_{1}\\
\nabla_{e_3}e_{1}=\mu_{3}e_2, & \nabla_{e_3}e_{2}=-\mu_{3}e_{1} & \nabla_{e_3}e_{3}=0.
\end{array}
\]
The Riemannian curvature $R$ is given by
\[
R(e_1,e_2)e_1=(\mu_{1}\mu_{2}-c_{3}\mu_{3})e_{2},\ \ 
R(e_1,e_2)e_2=-(\mu_{1}\mu_{2}-c_{3}\mu_{3})e_{1},\ \ 
\]
\[
R(e_2,e_3)e_2=(\mu_{2}\mu_{3}-c_{1}\mu_{1})e_{3},\ \ 
R(e_2,e_3)e_3=-(\mu_{2}\mu_{3}-c_{1}\mu_{1})e_{2},\ \ 
\]
\[
R(e_1,e_3)e_1=(\mu_{3}\mu_{1}-c_{2}\mu_{2})e_{3},\ \ 
R(e_1,e_3)e_3=-(\mu_{3}\mu_{1}-c_{2}\mu_{2})e_{1}. 
\]
The basis $\{e_1,e_2,e_3\}$ diagonalizes 
the Ricci tensor field. The principal Ricci curvatures 
are
\[
\rho_{1}=2\mu_{2}\mu_{3},
\ \ 
\rho_{2}=2\mu_{3}\mu_{1},
\ \ 
\rho_{3}=2\mu_{1}\mu_{2}.
\] 
\end{proposition}

\begin{example}[The space $\mathrm{Nil}_3$]
{\rm 
The model space $\mathrm{Nil}_3$ 
of nilgeometry in the sense of Thurston is
realized as a nilpotent Lie group 
\[
G=\left\{
\left(
\begin{array}{ccc}
1 & x & z\\
0 & 1 & y\\
0 &  0 & 1
\end{array}
\right)\>
\biggr\vert\>x,y,z\in\mathbb{R}
\right\}
\]
equipped with left invariant metric 
$dx^2+dy^2+(dz-xdy)^2$. 
The Lie algebra $\mathfrak{g}$ of $G$ is given by
\[
\mathfrak{g}=\left\{
\left(
\begin{array}{ccc}
0 & u & w\\
0 & 0 & v\\
0 &  0 & 0
\end{array}
\right)\>
\biggr\vert\>u,v,w\in\mathbb{R}
\right\}.
\]
Take an orthonormal basis $\{e_1,e_2,e_3\}$ 
of $\mathfrak{g}$: 
\[
e_1=
\left(
\begin{array}{ccc}
0 & 0 & 1\\
0 & 0 & 0\\
0 & 0 & 0
\end{array}
\right),
\>
e_2=
\left(
\begin{array}{ccc}
0 & 1 & 0\\
0 & 0 & 0\\
0 & 0 & 0
\end{array}
\right),
\>
e_3=
\left(
\begin{array}{ccc}
0 & 0 & 0\\
0 & 0 & 1\\
0 & 0 & 0
\end{array}
\right).
\]
Then $\{e_1,e_2,e_3\}$ is a unimodular basis 
satisfying $(c_1,c_2,c_3)=(1,0,0)$. 
The left invariant vector fields determined by 
$e_1$, $e_2$ and $e_3$ are given by 
\[
e_1=\frac{\partial}{\partial z},
\ \ 
e_2=\frac{\partial}{\partial x},
\ \
e_3=\frac{\partial}{\partial y}+x\frac{\partial}{\partial z}.
\]
}
\end{example}

\begin{example}[The space $\mathrm{Sol}_3$]
{\rm 
The model space $\mathrm{Sol}_3$ 
of solvegeometry in the sense of Thurston is
realized as a solvable Lie group
\[
G=\left\{
\left(
\begin{array}{ccc}
e^{-z} & 0 & x\\
0 & e^{z} & y\\
0 &  0 & 1
\end{array}
\right)\>
\biggr\vert\>x,y,z\in\mathbb{R}
\right\}
\]
equipped with left invariant metric 
$e^{2z}dx^2+e^{-2z}dy^2+dz^2$.
The Lie algebra $\mathfrak{g}$ of $G$ is given by
\[
\mathfrak{g}=\left\{
\left(
\begin{array}{ccc}
-w & 0 & u\\
0 & w & v\\
0 &  0 & 0
\end{array}
\right)\>
\biggr\vert\>u,v,w\in\mathbb{R}
\right\}.
\]
Take an orthonormal basis $\{e_1,e_2,e_3\}$ 
of $\mathfrak{g}$:
\[
e_1=\frac{1}{\sqrt{2}}(\hat{e}_1+\hat{e}_{2}),
\ \ 
e_2=\frac{1}{\sqrt{2}}(\hat{e}_1-\hat{e}_{2}),
\ \ 
e_3=\hat{e}_3,
\]
where 
\[
\hat{e}_1=
\left(
\begin{array}{ccc}
-1 & 0 & 0\\
0 & 1 & 0\\
0 & 0 & 0
\end{array}
\right),
\>
\hat{e}_2=
\left(
\begin{array}{ccc}
0 & 0 & 0\\
0 & 0 & 1\\
0 & 0 & 0
\end{array}
\right),
\>
\hat{e}_3=
\left(
\begin{array}{ccc}
0 & 0 & 1\\
0 & 0 & 0\\
0 & 0 & 0
\end{array}
\right)
\]
Then $\{e_1,e_2,e_3\}$ is a unimodular basis 
satisfying $(c_1,c_2,c_3)=(1,-1,0)$. 
The left translated vector fields of 
$e_1$, $e_2$ and $e_3$ are
\[
e_{1}=\frac{1}{\sqrt{2}}
\left(
e^{-z}\frac{\partial}{\partial x}
+
e^{z}\frac{\partial}{\partial y}
\right),
\ \ 
e_{2}=\frac{1}{\sqrt{2}}
\left(
e^{-z}\frac{\partial}{\partial x}
-
e^{z}\frac{\partial}{\partial y}
\right),
\ \ 
e_3=\frac{\partial}{\partial z}.
\]
}
\end{example}

\subsection{Magnetic equation}

Consider a left invariant unit vector field $X=x_1e_1+x_2e_2+x_3e_3$ on $G$, where $x_1, x_2, x_3$
are constants such that $x_1^2+x_2^2+x_3^2=1$.

We have to develop the two magnetic equations 
\eqref{Meq} provided in Theorem \ref{thm4.1}.

\textbf{1.} $\mathrm{tr}_g\,R(\nabla_{\cdot}X,X)\cdot=q\nabla_XX$

The left side of the equation above can be developed as follows
\begin{equation*}
\begin{array}{rl}
\mathrm{tr}\,R(\nabla_\cdot X,X)\cdot  = &
R(\nabla_{e_1}X,X)e_1+R(\nabla_{e_2}X,X)e_2+R(\nabla_{e_3}X,X)e_3\\[2mm]
  = &
\sum\limits_{(1,2,3)}
 \left[\mu_1(\mu_3^2-\mu_2^2)+(c_3-c_2)\mu_2\mu_3\right]x_2x_3e_1,
 \end{array}
\end{equation*}
where $\sum\limits_{(1,2,3)}$ denotes the cyclic summation over the permutation $(1,2,3)$.

On the other hand, we have
$\nabla_XX=\sum\limits_{(1,2,3)}(\mu_2-\mu_3)x_2x_3e_1$.

Using that $\mu_2-\mu_3=c_3-c_2$ and $\mu_2+\mu_3=c_1$, together with other four identities obtained by cyclic permutations, 
the first magnetic equation yields
\begin{equation}
\label{uni-MEq1}
\left\{
\begin{array}{l}
\left(\mu_1c_1-\mu_2\mu_3+q\right)(c_2-c_3)x_2x_3=0,\\[2mm]
\left(\mu_2c_2-\mu_3\mu_1+q\right)(c_3-c_1)x_3x_1=0,\\[2mm]
\left(\mu_3c_3-\mu_1\mu_2+q\right)(c_1-c_2)x_1x_2=0.
\end{array} \right.
\end{equation}

\textbf{ 2.} $\overline\Delta_gX=|\nabla X|^2X$

We compute first
\begin{equation*}
\overline\Delta_gX=-\sum\limits_{i=1}^3\left(\nabla_{e_i}\nabla_{e_i}X-\nabla_{\nabla_{e_i}e_i}X\right)
=\sum\limits_{(1,2,3)}(\mu_2^2+\mu_3^2)x_1e_1
\end{equation*}
and
\begin{equation*}
|\nabla X|^2=g(X,\overline{\Delta}_gX)=x_1^2(\mu_2^2+\mu_3^2)+x_2^2(\mu_3^2+\mu_1^2)+x_3^2(\mu_1^2+\mu_2^2).
\end{equation*}
The second magnetic equation yields
\begin{equation}
\label{uni-MEq2}
\left\{
\begin{array}{l}
\left(\mu_2^2+\mu_3^2-|\nabla X|^2\right)x_1=0,\\[2mm]
\left(\mu_3^2+\mu_1^2-|\nabla X|^2\right)x_2=0,\\[2mm]
\left(\mu_1^2+\mu_2^2-|\nabla X|^2\right)x_3=0.
\end{array}\right.
\end{equation}

Discussion:

{\bf (i)} two $x_i$'s are zero, i.e. $X=\pm e_1$ or $X=\pm e_2$ or $X=\pm e_3$

Equations \eqref{uni-MEq1} and \eqref{uni-MEq2} are trivially satisfied for any strength $q$.

{\bf (ii)} only one $x_i$ is zero, let's say $x_3=0$ and $x_1,x_2\neq0$ with $x_1^2+x_2^2=1$

On one hand, the equation \eqref{uni-MEq1} leads to $(c_1-c_2)(\mu_3c_3-\mu_1\mu_2+q)=0$ and 
on the other hand, the equation \eqref{uni-MEq2} implies $\mu_1^2=\mu_2^2$.

\begin{itemize}
\item $\mu_1=\mu_2$ is equivalent to $c_1=c_2$; in this situation, both magnetic equations are satisfied for any $q$;
\item $\mu_2=-\mu_1$ is equivalent to $c_3=0$; 
 \begin{itemize}
   \item[$\star$] if $c_1\neq c_2$, then the equation \eqref{uni-MEq1} is satisfied if and only if $q=-\frac{(c_2-c_1)^2}{4}$;
   \item[$\star$] if $c_1=c_2$ then $\mu_1=\mu_2=0$ and equation \eqref{uni-MEq1} is automatically satisfied for any $q$.
 \end{itemize}
\end{itemize}

{\bf (iii)} none of $x_i$'s is zero

Now, the equation \eqref{uni-MEq2} implies $\mu_1^2=\mu_2^2=\mu_3^2$:
\begin{itemize}
\item $\mu_1=\mu_2=\mu_3$ is equivalent to $c_1=c_2=c_3$, case when the equation \eqref{uni-MEq1} is satisfied for any $q$;
\item $\mu_1=\mu_2=-\mu_3\neq0$ is equivalent to $c_1=c_2=0$ and $c_3\neq0$; in this situation the equation \eqref{uni-MEq1}
is satisfied for $q=-\frac{c_3^2}{4}$;
\item the other cases are obtained from the previous one by cyclic permutations.
\end{itemize}


\begin{theorem}
Let $G$ be a $3$-dimensional unimodular Lie group with a left invariant metric and let $\{e_1, e_2, e_3\}$ be an orthonormal
basis of the Lie algebra $\mathfrak{g}$ satisfying \eqref{eq:unimodularbasis} with $c_1\geq c_2\geq c_3$. 
Then the unit magnetic vector fields on $G$ belong to the following list
\begin{center}\rm
\begin{tabular}{|l|c|c|c|}
\Xhline{3\arrayrulewidth}   
Conditions for $c_i$ & $G$ & The set of all unit    & $q$ \\
                     &     & magnetic vector fields & \\
\Xhline{3\arrayrulewidth}   
$c_1=c_2=c_3$ ${}\neq0$ & ${\mathbb{S}}^3$ & $\mathcal{S}$ & $\forall q$ \\[1mm]
            \hfill{  ${}=0$} &  ${\mathbb{R}}^3$     & $\mathcal{S}$ & $\forall q$ \\[1mm]
\hline   
$c_1>c_2=c_3$ ${}\neq0$ & $\mathrm{SU}(2)$, $\mathrm{SL}_2{\mathbb{R}}$ & $\pm e_1$, 
									${\mathcal{S}}\cap\{e_2,e_3\}_{\mathbb{R}}$  & $\forall q$ \\[1mm]
           \hfill{   ${}=0$} &  Heisenberg group & $\pm e_1$, 
                                    ${\mathcal{S}}\cap\{e_2,e_3\}_{\mathbb{R}}$ & $\forall q$  \\[1mm]
                      &                  &                 ${\mathcal{S}}$            & $-\dfrac{c_1^2}{4}$\\[2mm]
\hline
${}0\neq$ $c_1=c_2>c_3$ & $\mathrm{SU}(2)$, $\mathrm{SL}_2{\mathbb{R}}$ & $\pm e_3$, 
                                    ${\mathcal{S}}\cap\{e_1,e_2\}_{\mathbb{R}}$& $\forall q$ \\[1mm]
                 ${}0=$ &  Heisenberg group & $\pm e_1$, 
                                    ${\mathcal{S}}\cap\{e_2,e_3\}_{\mathbb{R}}$ & $\forall q$  \\[1mm]
                      &                  &                 ${\mathcal{S}}$            & $-\dfrac{c_3^2}{4}$\\[2mm]
\hline
$c_1>c_2>c_3>0$ & $\mathrm{SU}(2)$ & $\pm e_1$, $\pm e_2$ $\pm e_3$ & $\forall q$ \\[1mm]
$c_1>c_2>0>c_3$ & $\mathrm{SL}_2{\mathbb{R}}$ & $\pm e_1$, $\pm e_2$ $\pm e_3$ & $\forall q$ \\[1mm]
$c_1>0>c_2>c_3$ & $\mathrm{SL}_2{\mathbb{R}}$ & $\pm e_1$, $\pm e_2$ $\pm e_3$ & $\forall q$ \\[1mm]
$0>c_1>c_2>c_3$ & $\mathrm{SU}(2)$ & $\pm e_1$, $\pm e_2$ $\pm e_3$ & $\forall q$ \\[1mm]
\hline
$c_1>c_2>c_3=0$ & $\mathrm{E}(2)$ & $\pm e_1$, $\pm e_2$ $\pm e_3$ & $\forall q$ \\[1mm]
                &        & $\pm e_3$, 
                           ${\mathcal{S}}\cap\{e_1,e_2\}_{\mathbb{R}}$ & $-\frac{(c_1-c_2)^2}{4}$\\[2mm]
\hline                        
$c_1>c_2=0>c_3$ & $\mathrm{E}(1,1)$ & $\pm e_1$, $\pm e_2$ $\pm e_3$ & $\forall q$ \\[1mm]
                &        & $\pm e_2$, 
						  ${\mathcal{S}}\cap\{e_1,e_3\}_{\mathbb{R}}$ & $-\frac{(c_3-c_1)^2}{4}$\\[2mm]
\hline
$0=c_1>c_2>c_3$ & $\mathrm{E}(2)$ & $\pm e_1$, $\pm e_2$ $\pm e_3$ & $\forall q$ \\[1mm]
                &        & $\pm e_1$, 
						  ${\mathcal{S}}\cap\{e_2,e_3\}_{\mathbb{R}}$ & $-\frac{(c_2-c_3)^2}{4}$\\[2mm]

\Xhline{3\arrayrulewidth}   
\end{tabular}
\end{center}
Here $\mathcal{S}$ is the unit sphere in the Lie algebra $\mathfrak{g}$ of $G$ centered at the origin.
\end{theorem}


\begin{remark}[Contact metric structure]
{\rm On a $3$-dimensional \emph{non-abelian} unimodular Lie group 
$G=G(c_1,c_2,c_3)$ with $c_1=2$, we can introduce a left 
invariant almost contact structure $(\varphi,\xi,\eta)$ compatible with 
the left invariant metric $\langle\cdot,\cdot\rangle$ by \cite{I, Perrone1997}:
\[
\xi=e_{1},\ \ \eta=g(\xi,\cdot),\ \ 
\varphi e_1=0,\ \ 
\varphi e_2=e_3,\ \ 
\varphi e_3=-e_2.
\] 
Then one can check that $(\varphi,\xi,\eta,\langle\cdot,\cdot\rangle)$ is a 
left invariant contact metric structure on $G$.
The $\varphi$-sectional curvature of $G$ is constant 
\[
-3+\frac{1}{4}(c_2-c_3)^2+c_2+c_3.
\]
In particular $G$ is Sasakian when and only when 
$c_2=c_3$. 

Moreover if $c_2\not=c_3$, $G$ is a 
contact $(\kappa,\mu)$-space with 
\[
\kappa=1-\frac{1}{4}(c_2-c_3)^2,\ \ 
\mu=2-(c_2+c_3).
\]
}
\end{remark}
Perrone \cite{Perrone2008} investigated stability of the Reeb vector
fields on 3-dimensional compact contact $(\kappa,\mu)$-spaces with
respect to the Dirichlet energy. For example, in \cite[Example~3.1]{Perrone2008}
the author provides (probably) the first examples of non-Killing harmonic vector fields which are
stable.

\section{Non-unimodular Lie groups}
\subsection{}
Let $G$ be a non-unimodular $3$-dimensional Lie group with a left invariant metric. 
Then the \emph{unimodular kernel} $\mathfrak{u}$ of $\mathfrak{g}$ is defined by
\[
\mathfrak{u}=\{X \in \mathfrak{g}
\ \vert \
\mathrm{tr}\> \mathrm{ad}(X)=0\}.
\]
Here $\mathrm{ad}:\mathfrak{g}\to \mathrm{End}(\mathfrak{g})$ is 
a homomorphism defined by
\[
\mathrm{ad}(X)Y=[X,Y].
\]
One can see that $\mathfrak{u}$ is an ideal of $\mathfrak{g}$ which contains 
the ideal $[\mathfrak{g},\mathfrak{g}]$.

On $\mathfrak{g}$, we can take an orthonormal basis
$\{e_1,e_2,e_3\}$ such that
\begin{enumerate}
\item $\langle e_{1}, X\rangle=0,\ X \in \mathfrak{u}$,
\item $\langle [e_1,e_2], [e_1,e_3]\rangle=0$.
\end{enumerate} 
Then the commutation relations of the basis are given by
\begin{equation}
\label{Lie-bra-n-u}
[e_1,e_2]=a_{11} e_{2}+a_{12} e_{3},\
[e_2,e_3]=0,\ 
[e_1,e_3]=a_{21} e_2+a_{22} e_3,
\end{equation}
with $a_{11}+a_{22}\not=0$ and $a_{11}a_{21}+a_{12}a_{22}=0$.
Under a suitable homothetic change of the metric, we may assume that
$a_{11}+a_{22}=2$. 
Moreover, we may assume that $a_{11}a_{22}+a_{12}a_{21}=0$ \cite[\S 2.5]{MP}. 
Then the constants $a_{11}$,
$a_{12}$, $a_{21}$ and $a_{22}$ are represented
as 
\begin{equation}
\label{n-u-coeff}
a_{11}=1+\alpha,\ a_{12}=(1+\alpha)\beta,\
a_{21}=-(1-\alpha)\beta,\
a_{22}=1-\alpha. 
\end{equation}
If necessarily, 
by changing the sign of $e_1$, $e_2$ and 
$e_3$, we may assume that the constants $\alpha$ and 
$\beta$ satisfy the condition $\alpha$, $\beta\geq 0$.
We note that for the case that $\alpha=\beta=0$, $G$ is of constant negative 
curvature (see Example \ref{eg:6.1}). 
We refer $(\alpha,\beta)$ as the structure constants 
of the non-unimodular Lie algebra $\mathfrak{g}$.
Non-unimodular Lie algebras $\mathfrak{g}=\mathfrak{g}(\alpha,\beta)$ are
classified by the \emph{Milnor invariant} $\mathcal{D}=(1-\alpha^2)(1+\beta^2)$.

\begin{proposition}
 For any $(\alpha,\beta)\not=(0,0)$, 
 two Lie algebras $\mathfrak{g}(\alpha,\beta)$ and 
$\mathfrak{g}(\alpha^\prime,\beta^\prime)$ are isomorphic
if and only if their Milnor invariants $\mathcal{D}$ and $\mathcal{D}^\prime$ 
agree. 
\end{proposition}

Under this normalization, 
the Levi-Civita connection of $G$ is given by the following table:
\begin{proposition}
\[
\begin{array}{ccc}
\nabla_{e_1}e_{1}=0, & \nabla_{e_1}e_{2}=\beta e_{3}, & \nabla_{e_1}e_{3}=-\beta e_{2},\\
\nabla_{e_2}e_{1}=-(1+\alpha)e_{2}-\alpha\beta e_{3}, 
& \nabla_{e_2}e_{2}=(1+\alpha) e_1, & 
\nabla_{e_2}e_{3}=\alpha\beta e_{1},\\
\nabla_{e_3}e_{1}=-\alpha\beta e_{2}
-(1-\alpha)e_{3}, 
& \nabla_{e_3}e_{2}=\alpha\beta e_{1}, & \nabla_{e_3}e_{3}=(1-\alpha)e_1.
\end{array}
\]
The Riemannian curvature $R$ is given by
\begin{align*}
R(e_1,e_2)e_1=&\{\alpha\beta^{2}+(1+\alpha)^{2}+\alpha\beta^{2}(1+\alpha)\}e_{2},
\\
R(e_1,e_2)e_2=&-\{\alpha\beta^{2}+(1+\alpha)^{2}+\alpha\beta^{2}(1+\alpha)\}e_{1},
\\
R(e_1,e_3)e_1=&-\{\alpha\beta^{2}-(1-\alpha)^{2}+\alpha\beta^{2}(1-\alpha)\}e_{3},
\\
R(e_1,e_3)e_3=&\{\alpha\beta^{2}-(1-\alpha)^{2}+\alpha\beta^{2}(1-\alpha)\}e_{1},
\\
R(e_2,e_3)e_2=&\{1-\alpha^{2}(1+\beta^{2})\}e_{3},
\\
R(e_2,e_3)e_3=&-\{1-\alpha^{2}(1+\beta^{2})\}e_{2}.
\end{align*}
The basis $\{e_1,e_2,e_3\}$ diagonalizes the Ricci tensor field.
The principal Ricci curvatures are given by
\[
\rho_{1}=-2\{1+\alpha^2(1+\beta^2)\}<-2,\ \ 
\rho_{2}=-2\{1+\alpha(1+\beta^2)\}<-2,\ \ 
\rho_{3}=-2\{1-\alpha(1+\beta^2)\}.
\]
The scalar curvature is 
\[
-2\{3+\alpha^2(1+\beta^2)\}<0.
\]
\end{proposition}
The simply connected Lie 
group $\widetilde{G}=\widetilde{G}(\alpha,\beta)$ corresponding to  
the non-unimodular Lie algebra $\mathfrak{g}(\alpha,\beta)$ 
is given explicitly by \cite{IN}:
\[
\widetilde{G}(\alpha,\beta)=
\left\{
\left(
\begin{array}{cccc}
1 & 0 & 0 & x\\
0 & \alpha_{11}(x) & \alpha_{12}(x) & y\\
0 & \alpha_{21}(x) & \alpha_{22}(x) & z\\
0 & 0 & 0 &1
\end{array}
\right)
\
\biggr
\vert
\
x,y,z \in \mathbb{R}
\right\},
\]
where 
\[
\left(
\begin{array}{cc}
\alpha_{11}(x) & \alpha_{12}(x)\\
\alpha_{21}(x) & \alpha_{22}(x)
\end{array}
\right)=
\exp\left(
x
\left[
\begin{array}{cc}
1+\alpha & (1+\alpha)\beta\\
-(1-\alpha)\beta & 1-\alpha
\end{array}
\right]
\right).
\]
This shows that $\widetilde{G}(\alpha,\beta)$ is the semi-direct product
$\mathbb{R}\ltimes\mathbb{R}^2$ 
with multiplication
\begin{equation*}
\label{semi-direct}
(x,y,z)\cdot (x^\prime,y^\prime,z^\prime)=
(x+x^\prime,y+\alpha_{11}(x)y^\prime+\alpha_{12}(x)z^\prime,
z+\alpha_{21}(x)y^\prime+\alpha_{22}(x)z^\prime).
\end{equation*}
The Lie algebra of $\widetilde{G}(\alpha,\beta)$ is 
is spanned by the basis
\[
e_{1}=
\left(
\begin{array}{cccc}
0 & 0 & 0 & 1\\
0 & 1+\alpha & -(1-\alpha)\beta & 0\\
0 & (1+\alpha)\beta & 1-\alpha & 0\\
0 & 0 & 0 &0
\end{array}
\right),
\ \
e_{2}=
\left(
\begin{array}{cccc}
0 & 0 & 0 & 0\\
0 & 0 & 0 & 1\\
0 & 0 & 0 & 0\\
0 & 0 & 0 &0
\end{array}
\right),
\ \
e_{3}=
\left(
\begin{array}{cccc}
0 & 0 & 0 & 0\\
0 & 0 & 0 & 0\\
0 & 0 & 0 & 1\\
0 & 0 & 0 &0
\end{array}
\right).
\]
This basis satisfies the commutation relation
\[
[e_1,e_2]=(1+\alpha)(e_{2}+\beta e_{3}),\ \
[e_2,e_3]=0,\ \ 
[e_3,e_1]=(1-\alpha)(\beta e_2-e_{3}).
\]
Thus the Lie algebra of $\widetilde{G}(\alpha,\beta)$ 
is the non-unimodular Lie algebra 
$\mathfrak{g}=\mathfrak{g}(\alpha,\beta)$. 
The left invariant vector fields corresponding to 
$e_1$, $e_2$ and $e_3$ are
\[
e_1=\frac{\partial}{\partial x},\ \ 
e_2=\alpha_{11}(x)\frac{\partial}{\partial y}+\alpha_{12}(x)\frac{\partial}{\partial z},
\ \
e_3=\alpha_{21}(x)\frac{\partial}{\partial y}+\alpha_{22}(x)\frac{\partial}{\partial z}.
\]
\begin{example}[$\alpha=0$, $\mathcal{D}\geq 1$]\label{eg:6.1}
{\rm The simply connected Lie group 
$\widetilde{G}(0,\beta)$ is isometric to the hyperbolic $3$-space 
$\mathbb{H}^3(-1)$ of curvature $-1$ and given 
explicitly by
\[
\widetilde{G}(0,\beta)=
\left\{
\left(
\begin{array}{cccc}
1 & 0 & 0 & x\\
0 & e^{x}\cos(\beta x) & -e^{x}\sin(\beta x) & y\\
0 & e^{x}\sin(\beta x) & e^{x}\cos(\beta x) & z\\
0 & 0 & 0 &1
\end{array}
\right)
\
\biggr
\vert
\
x,y,z \in \mathbb{R}
\right\}. 
\]
The left invariant metric is 
$dx^2+e^{-2x}(dy^2+dz^2)$. 
Thus $\widetilde{G}(0,\beta)$
is the warped product model of $\mathbb{H}^{3}(-1)$.
In fact, setting $w=e^{x}$, the left invariant metric of $\widetilde{G}$ 
can be rewritten as the Poincar{\'e} metric
\[
\frac{dy^2+dz^2+dw^2}{w^2}.
\]
The Milnor invariant of $\widetilde{G}(0,\beta)$ is 
$\mathcal{D}=1+\beta^2\geq 1$.
}
\end{example}

\begin{example}[$\beta=0$, $\mathcal{D}\leq1$]\label{eg:6.2}{\rm
For each $\alpha\geq 0$, $\widetilde{G}(\alpha,0)$ is given by:
\[
\widetilde{G}(\alpha,0)=\left\{
\left(
\begin{array}{cccc}
1 & 0 & 0 & x\\
0 & e^{(1+\alpha)x} & 0 & y\\
0 & 0 & e^{(1-\alpha)x} & z\\
0 & 0 & 0 &1
\end{array}
\right)
\
\biggr
\vert
\
x,y,z \in \mathbb{R}
\right\}.
\]
The left invariant Riemannian metric is given explicitly by
\[
dx^{2}+e^{-2(1+\alpha)x}dy^{2}+e^{-2(1-\alpha)x}dz^{2}.
\]
The Milnor invariant is $\mathcal{D}=1-\alpha^2\leq 1$.

This family of Riemannian homogenous spaces has been studied in
\cite{I-L, Nistor}. The Lie group $\widetilde{G}(\alpha,0)$ 
is realized as a warped product of 
the hyperbolic plane $\mathbb{H}^{2}(-(1+\alpha)^2)$
of curvature $-(1+\alpha)^2$ and the real line $\mathbb{R}$ with 
warping function $e^{(\alpha-1)x}$. In fact, via the coordinate change 
$(u,v)=(\,(1+\alpha)y,e^{(1+\alpha)x})$, the metric is rewritten as
\[
\frac{du^2+dv^2}{(1+\alpha)^2v^2}+f_{\alpha}(u,v)^2\, dz^2, \ \ f_{\alpha}(u,v)=
\exp\left(
\frac{\alpha-1}{\alpha+1}\log v
\right).
\]
Here we observe locally symmetric examples:
\begin{itemize}
\item If $\alpha=0$ then $\widetilde{G}(0,0)$ is a warped product model of 
hyperbolic 3-space $\mathbb{H}^{3}(-1)$. 
\item If $\alpha=1$ then $\widetilde{G}(1,0)$ is 
isometric to the Riemannian product $\mathbb{H}^{2}(-4)\times \mathbb{R}$. 
\end{itemize}
Note that $\mathbb{H}^3$ does not admit
any other Lie group structure. 
}
\end{example}

\begin{example}[$\alpha=1$, $\mathcal{D}=0$]\label{eg:6.3}{\rm
Assume that $\alpha=1$. Then $\widetilde{G}(1,\beta)$ is given explicitly by
\[
\widetilde{G}(1,\beta)=
\left\{
\left(
\begin{array}{cccc}
1 & 0 & 0 & x\\
0 & e^{2x} & 0 & y\\
0 & \beta(e^{2x}-1) & 1 & z\\
0 & 0 & 0 &1
\end{array}
\right)
\
\biggr
\vert
\
x,y,z \in \mathbb{R}
\right\}.
\]
The left invariant metric is 
\[
dx^2+\{e^{-4x}+\beta^2(1-e^{-2x})^2\}dy^2-2\beta(1-e^{-2x})dydz+dz^2.
\]
The non-unimodular Lie group $\widetilde{G}(1,\beta)$ has sectional curvatures
\[
K_{12}=-3\beta^{2}-4,\ 
\ 
K_{13}=K_{23}=\beta^{2},
\]
where $K_{ij}$ ($i\not=j$) denote the 
sectional curvatures of the planes spanned
by vectors $e_i$ and $e_j$. 
One can check that $\widetilde{G}(1,\beta)$ is isometric to
the so-called 
\emph{Bianchi-Cartan-Vranceanu space} $M^{3}(-4,\beta)$
with $4$-dimensional isometry group and 
isotropy subgroup $\mathrm{SO}(2)$:
\[
M^{3}(-4,\beta)=
\left(
\{(u,v,w)\in \mathbb{R}^{3}
\
\vert 
\
u^2+v^2<1\},g_{_\beta}
\right),
\]
with metric 
\[
g_{_\beta}=\frac{du^2+dv^2}{(1-u^2-v^2)^{2}}+
\left(
dw+ \frac{\beta(vdu-udv)}{1-u^2-v^2}
\right)^{2},\ \ \beta\geq 0.
\]
The family $\{\widetilde{G}(1,\beta)\}_{\beta\geq 0}$ is 
characterized by the condition $\mathcal{D}=0$. 
In particular $M^{3}(-4,\beta)$ with \emph{positive} $\beta$ is isometric to 
the universal covering $\widetilde{\mathrm{SL}}_{2}\mathbb{R}$
of the special linear group equipped with the above metric 
(\textit{cf.} \cite{Thurston}), 
but 
\emph{not
isomorphic} to $\widetilde{\mathrm{SL}}_{2}\mathbb{R}$
as Lie groups. We here note that $\widetilde{\mathrm{SL}}_{2}\mathbb{R}$ is a 
\emph{unimodular} Lie group, while $\widetilde{G}(1,\beta)$ is
\emph{non-unimodular}. 
}
\end{example}

\subsection{Magnetic equations}
We first recall that $\alpha,\beta\geq0$
and make some notations in order to simplify some other complicated expressions
\begin{equation*}
\left\{
\begin{array}{l}
u=\alpha\beta^2+(1+\alpha)^2+\alpha\beta^2(1+\alpha),\\[2mm]
v=\alpha\beta^2-(1-\alpha)^2+\alpha\beta^2(1-\alpha),\\[2mm]
w=1-\alpha^2(1+\beta^2).
\end{array}\right.
\end{equation*}
Take a left invariant unit vector field $X=x_1e_1+x_2e_2+x_3e_3$ on $G$, where $x_1, x_2, x_3$
are constants such that $x_1^2+x_2^2+x_3^2=1$.

We have to develop the two magnetic equations provided in Theorem~\ref{thm4.1}. 

{\bf 1.} ${\rm tr}_g R(\nabla_{\cdot}X,X)\cdot=q\nabla_XX$

After a straightforward computation we obtain
\begin{subequations}
\label{non-uni-MEq1}
\begin{equation}
\label{n-u-1.1}
\begin{array}{l}
-\left[x_2(1+\alpha)+x_3\alpha\beta\right]x_2u-x_1^2(1+\alpha)u+\left[x_2\alpha\beta+
   x_3(1-\alpha)\right]x_3v \\[1mm]
   \qquad\qquad\qquad
   +x_1^2(1-\alpha)v =q\left[x_2^2(1+\alpha)+x_3^2(1-\alpha)+x_2x_3\alpha\beta\right]
\end{array}
\end{equation}
\begin{equation}
\label{n-u-1.2}
x_1x_3\beta u+x_1x_3\beta w = -q\left[x_1x_3\beta(1+\alpha)+x_1x_2(1+\alpha)\right]
\end{equation}
\begin{equation}
\label{n-u-1.3}
\begin{array}{l}
x_1x_2\beta v-x_1x_3(1+\alpha)w+x_1x_2\alpha\beta w+x_1x_2(1-\alpha)w \\[1mm]
\hfill	=q\left[x_1x_2\beta(1-\alpha)-x_1x_3(1-\alpha)\right].
\end{array}
\end{equation}
\end{subequations}

{\bf 2.} $\overline\Delta_gX=|\nabla X|^2X$

Again, after a straightforward computations, we get:
\begin{subequations}
\label{non-uni-MEq2}
\begin{equation}
\label{n-u-2.1}
2x_1\big(1+\alpha^2+\alpha^2\beta^2\big)=|\nabla X|^2 x_1
\end{equation}
\begin{equation}
\label{n-u-2.2}
x_2\big[\beta^2+(1+\alpha)^2+\alpha^2\beta^2\big]-2x_3\beta(1-\alpha)=|\nabla X|^2x_2
\end{equation}
\begin{equation}
\label{n-u-2.3}
x_3\big[\beta^2+(1-\alpha)^2+\alpha^2\beta^2\big]+2x_2\beta(1+\alpha)=|\nabla X|^2x_3,
\end{equation}
\end{subequations}
where
\begin{equation}
\begin{array}{c}
|\nabla X|^2=g(\overline\Delta_gX,X) = ~ 2x_1^2\big(1+\alpha^2+\alpha^2\beta^2\big)\\[1mm]
   +x_2^2\big[\beta^2+(1+\alpha)^2+\alpha^2\beta^2\big]
    +x_3^2 \big[\beta^2+(1-\alpha)^2+\alpha^2\beta^2\big]+4x_2x_3\alpha\beta.
\end{array}
\end{equation}

Recall that $\pm e_1$, $\pm e_2$ and $\pm e_3$ are all unit magnetic vector fields in the unimodular
Lie groups. Contrary, the non-unimodular case is more rigid. For example, we have the following:
\begin{center}
\begin{tabular}{|c|l|c|}
\Xhline{3\arrayrulewidth}  
 & When $e_i$ is magnetic? & Condition for $q$\\
\Xhline{3\arrayrulewidth}  
$e_1$ & never & --\\
\hline
$e_2$ & $\beta=0$ & $q=-(1+\alpha)^2$ \\
& otherwise, never & --\\
\hline
$e_3$ & $\alpha=1$ & $\forall q$\\
 & $\alpha\neq1$ and $\beta=0$ & $q=-(1-\alpha)^2$\\
 & otherwise, never & --\\
 \hline
\Xhline{3\arrayrulewidth}  
\end{tabular}
\end{center}

Let us continue our investigation with some particular situations:

{\bf I.} $\alpha=0$ and $\beta\in{\mathbb{R}}$ \quad \fbox{$G$ is isometric to ${\mathbb{H}}^3(-1)$ }
\quad $\mathcal{D}\geq1$

We have $u=1$, $v=-1$, $w=1$ and $|\nabla X|^2=1+x_1^2+\beta^2(1-x_1^2)$.

The two magnetic equations simplify and we obtain
\begin{equation*}
\left\{
\begin{array}{rcl}
1+x_1^2&=&-q(1-x_1^2), \\[1mm]
0&=&-qx_1(x_2+\beta x_3), \\ [1mm]
-x_1x_2\beta-x_1x_3\beta+x_1x_2&=& ~ qx_1(\beta x_2-x_3) \\[2mm]
2x_1&=&|\nabla X|^2x_1 \\[1mm]
(1+\beta^2)x_2-2\beta x_3&=&|\nabla X|^2x_2\\[1mm]
(1+\beta^2)x_3+2\beta x_2&=&|\nabla X|^2x_3.
\end{array}
\right.
\end{equation*}

This system has solution if and only $\beta=0$. 
Hence we can state the following.
\begin{proposition} 
\label{n-u-P1}
Let $G(0,\beta)$ be a $3$-dimensional non-unimodular Lie group
with left invariant metric and let $\{e_1,e_2,e_3\}$ be an orthonormal basis of the Lie algebra $\mathfrak{g}$
as described before. Then  
\begin{itemize}
\item[(i)] if $\beta\neq0$ there do not exist unit magnetic vector fields on $G$;
\item[(ii)] if $\beta=0$ then the unit magnetic vector fields on $G$ are $x_2e_2+x_3e_3$, 
with $x_2^2+x_3^2=1$, case when the strength $q=-1$.
\end{itemize}
\end{proposition}

{\bf II.} $\alpha=1$ and $\beta=0$ \quad \fbox{$G$ is isometric to ${\mathbb{H}}^2(-4)\times{\mathbb{R}}$}
\quad $\mathcal{D}=0$

We have $u=4$, $v=0$, $w=0$ and $|\nabla X|^2=4(x_1^2+x_2^2)$.

The two magnetic equations reduce to
\begin{equation*}
\left\{
\begin{array}{rcl}
-8x_2^2-8x_1^2&=&~2qx_2^2\\[1mm]
0&=&-2qx_1x_2\\[2mm]
4x_1&=&|\nabla X|^2 x_1\\[1mm]
4x_2&=&|\nabla X|^2 x_2\\[1mm]
0&=&|\nabla X|^2 x_3.
\end{array}
\right.
\end{equation*}
The following result holds.
\begin{proposition}
\label{n-u-P2}
Let $G(1,0)$ be a $3$-dimensional non-unimodular Lie group
with left invariant metric and let $\{e_1,e_2,e_3\}$ be an orthonormal basis of the Lie algebra $\mathfrak{g}$
as described before. Then the unit magnetic vector fields on $G$ are 
\begin{itemize}
\item[(i)] $\pm e_3$, for arbitrary $q$
\item[(ii)] $\pm e_2$, for $q=-4$.
\end{itemize}
\end{proposition}

{\bf III.} $\alpha=1$ and $\beta\neq0$ \quad 
\fbox{$G$ is isometric to the BCV space $M^3(-4,\beta)$}
\quad
$\mathcal{D}=0$

We have $u=3\beta^2+4$, $v=\beta^2$, $w=-\beta^2$ and $|\nabla X|^2=2\beta^2+4(x_1^2+x_2^2)+4x_2x_3\beta$.

The two magnetic equations reduce to
\begin{equation*}
\left\{
\begin{array}{rcl}
-(2x_2+\beta x_3)x_2(3\beta^2+4)-2x_1^2(3\beta^2+4)+x_2x_3\beta^3&=&~q(2x_2^2+\beta x_2x_3)\\[1mm]
2x_1x_3\beta(\beta^2+2)&=&-2qx_1(\beta x_3+x_2)\\[1mm]
2x_1x_3\beta^2 &=&0\\[2mm]
2x_1(\beta^2+2)&=&|\nabla X|^2 x_1\\[1mm]
2x_2(\beta^2+2)&=&|\nabla X|^2 x_2\\[1mm]
2x_3\beta^2+4x_2\beta&=&|\nabla X|^2 x_3.
\end{array}
\right.
\end{equation*}
We obtain the following.
\begin{proposition}
\label{n-u-P3}
Let $G(1,\beta)$ with $\beta\neq0$, be a $3$-dimensional non-unimodular Lie group
with left invariant metric and let $\{e_1,e_2,e_3\}$ be an orthonormal basis of the Lie algebra $\mathfrak{g}$
as described before. Then the unit magnetic vector fields on $G$ are 
\begin{enumerate}
\item $\pm e_3$, case when the strength $q$ is arbitrary,
\item
$  \pm\dfrac{1}{\sqrt{\beta^2+1}}(e_2+\beta e_3),
$
case when the strength $q=\dfrac{4}{\beta^2+2}-2\beta^2-6$. \\
{\rm ($q$ never vanishes.)}
\end{enumerate}
\end{proposition}

Before discussing the general situation, we ask,
as in the case when $G$ is unimodular, the following question:

\quad
{\sl When $\pm e_1$, $\pm e_2$ and $\pm e_3$ are, respectively
unit magnetic vector fields in $G(\alpha,\beta)$?}

{\bf Q1.} When $\pm e_1$ is magnetic?

Equation \eqref{n-u-1.1} implies $-(1+\alpha)u+(1-\alpha)v=0$, that is equivalent to
\newline
$1+3\alpha^2+3\alpha^2\beta^2=0$. But this is a contradiction.

{\bf A1.} The unit vector field $\pm e_1$ is never magnetic.

\medskip

{\bf Q2.} When $\pm e_2$ is magnetic?

The two magnetic equations \eqref{non-uni-MEq1} and \eqref{non-uni-MEq2} imply $q=-u$ and $\beta=0$.
\newline
We have ${\mathcal{D}}=1-\alpha^2\leq 1$.

{\bf A2.} The unit vector field $\pm e_2$ is magnetic only on $G(\alpha,0)$ and in this case the strength is $q=-(1+\alpha)^2$.

\medskip

{\bf Q3.} When $\pm e_3$ is magnetic?

{\bf A3.} The unit vector field $\pm e_3$ is magnetic on $G(\alpha,\beta)$ if and only if
\begin{itemize}
\item either $G=G(1,\beta)$ (case when ${\mathcal{D}}=0$ and we do not have any condition for $q$),
\item or $G=G(\alpha,0)$ with $\alpha\neq1$ and in this case ${\mathcal{D}}\leq1$ and the strength is $q=-(1-\alpha)^2$.
\end{itemize}

Finally, we analyze the general situation.

We are looking first for solutions $X=x_1e_1+x_2e_2+x_3e_3$ with $x_1\neq0$.

From \eqref{n-u-2.1} we find $|\nabla X|^2=2(1+\alpha^2+\alpha^2\beta^2)$. Replacing $|\nabla X|^2$ in equations
\eqref{n-u-2.2} and \eqref{n-u-2.3} we obtain
\begin{equation}
\left\{
\begin{array}{l}
\left[\beta^2-(1-\alpha)^2-\alpha^2\beta^2\right]x_2-2\beta(1-\alpha)x_3=0,\\[2mm]
2\beta(1+\alpha)x_2+\left[\beta^2-(1+\alpha)^2-\alpha^2\beta^2\right]x_3=0.
\end{array}\right.
\end{equation}
This is a homogeneous system of two linear equations whose determinant is ${\mathcal{D}}^2$.

So, if ${\mathcal{D}}\neq0$, we have only the trivial solution, i.e. $x_2=0$ and $x_3=0$. This implies that
$X=\pm e_1$. Bringing the question {\bf Q1} to the present moment, we note that this situation cannot occur. 
Therefore, a necessary condition to have a solution $X$ (with $x_1\neq0$) is ${\mathcal{D}}=0$, that is $\alpha=1$.
This situation is explained in details in Propositions \ref{n-u-P2} and \ref{n-u-P3}. 
Subsequently, we find that the solutions do not exist if $x_1\neq0$.

We are looking now for solutions $X=x_2e_2+x_3e_3$. The only non-trivial equations are obtained
from \eqref{n-u-1.1}, \eqref{n-u-2.2} and \eqref{n-u-2.3}. The last two yield
\begin{equation}
\label{n-u-gen-01}
\left\{
\begin{array}{l}
2\alpha x_2x_3(x_3-\beta x_2)-\beta(1-\alpha)x_3=0,\\[2mm]
2\alpha x_2x_3(x_2+\beta x_3)-\beta(1+\alpha)x_2=0.
\end{array}\right.
\end{equation}

The cases $x_2=0$ and $x_3=0$ are discussed in {\bf Q3} and {\bf Q2}, respectively.
Moreover, the case $\alpha=0$ is presented in Proposition~\ref{n-u-P1}.
Therefore, from now on we consider $\alpha\neq0$ and we suppose that both $x_2$ and $x_3$ do not vanish.

An immediate consequence is that $\beta$ cannot vanish. 

The equations \eqref{n-u-gen-01} imply
$$
\beta(1+\alpha)x_2^2-2\alpha x_2x_3+\beta(1-\alpha)x_3^2=0.
$$
This equation has solutions if and only if ${\mathcal{D}}\leq1$, namely
\begin{equation}
\label{n-u-gen-sol}
x_2=\frac{\varepsilon t}{\sqrt{1+t^2}}, \quad
x_3=\frac{\varepsilon}{\sqrt{1+t^2}}, \ {\rm where} \ \varepsilon=\pm1
\ {\rm and}\ 
t=\frac{\alpha\pm\sqrt{1-{\mathcal{D}}}}{\beta(1+\alpha)}.
\end{equation}
They are solutions also for \eqref{n-u-gen-01}. Finally, the strength $q$ is obtained from equation \eqref{n-u-1.1}:
\begin{equation}
\label{n-u-q-gen}
q=-\frac{2(1+\beta^2)\left(1+\alpha^2+\alpha^2\beta^2\pm2\sqrt{1-{\mathcal{D}}}\right)}{2+\beta^2}.
\end{equation} 
\begin{remark} The conclusion of Proposition~\ref{n-u-P3} is a particular case of the previous discussion,
that is choosing the non vanishing solution for $t$ in the case $\alpha=1$.
\end{remark}
\begin{remark} In the case when ${\mathcal{D}}=1$, the unit magnetic vector field is 
\newline
$\pm\left(\sqrt{\frac{1-\alpha}{2}}~e_2+\sqrt{\frac{1+\alpha}{2}}~e_3\right)$ and the strength is 
$q=-\frac{2}{(1-\alpha^2)(2-\alpha^2)}=-\frac{2(\beta^2+1)^2}{\beta^2+2}$.
\end{remark}

\newpage

\begin{theorem}
Let $G$ be a $3$-dimensional non-unimodular Lie group with a left invariant metric and let $\{e_1, e_2, e_3\}$ be an orthonormal
basis of the Lie algebra $\mathfrak{g}$ satisfying \eqref{Lie-bra-n-u} with \eqref{n-u-coeff}. 
Then the unit magnetic vector fields on $G$ belong to the following list
\begin{center} \rm
\begin{tabular}{|l|c|c|}
\Xhline{3\arrayrulewidth}   
 & &\\
\quad condition for ${\mathcal{D}}$ & the set of all unit & the strength\\
                  & magnetic vector fields on $G(\alpha,\beta)$ &\\
\Xhline{3\arrayrulewidth}   
${\mathcal{D}}>1$ & $\emptyset$ & -- \\[2mm]
\hline
${\mathcal{D}}=1$, $\alpha=0$, $\beta=0$ & ${\mathcal{S}}\cap\{e_2,e_3\}_{\mathbb{R}}$& $q=-1$\\[2mm]
${\mathcal{D}}=1$, $\alpha\in(0,1)$ $\beta=\frac{\alpha}{\sqrt{1-\alpha^2}}$  & 
   $\pm\left(\sqrt{\frac{1-\alpha}{2}}~e_2+\sqrt{\frac{1+\alpha}{2}}~e_3\right)$&
   $q=-\frac{2}{(1-\alpha^2)(2-\alpha^2)}$\\[3mm]
\hline
${\mathcal{D}}\in(0,1)$, $\alpha\in(0,1)$, $\beta=0$ & $\pm e_2$& $q=-(1+\alpha)^2$\\[2mm]
                                         & $\pm e_3$ & $q=-(1-\alpha)^2$\\[2mm]
${\mathcal{D}}\in(0,1)$, $\alpha\in(0,1)$, $\beta\neq0$ & $x_2e_2+x_3e_3$ with \eqref{n-u-gen-sol} & \eqref{n-u-q-gen} \\
\hline
${\mathcal{D}}=0$, $\alpha=1$, $\beta=0$ & $\pm e_3$ & $\forall q$ \\ [2mm]
                             & $\pm e_2$ & $q=-4$\\[2mm]
${\mathcal{D}}=0$, $\alpha=1$, $\beta\neq0$ & $\pm e_3$ & $\forall q$\\ [2mm]                            
 & $\pm\frac{1}{\sqrt{\beta^2+1}}(e_2+\beta e_3)$ &
            $q=\frac{4}{\beta^2+2}-2\beta^2-6$\\[2mm]
\hline
${\mathcal{D}}<0$, $\beta=0$ & $\pm e_2$ & $q=-(1+\alpha)^2$ \\[2mm]
                 & $\pm e_3$ & $q=-(1-\alpha)^2$\\[2mm]
${\mathcal{D}}<0$, $\beta\neq0$ & $x_2e_2+x_3e_3$ with \eqref{n-u-gen-sol} & \eqref{n-u-q-gen} \\
\Xhline{3\arrayrulewidth}   
\end{tabular}
\end{center}
\end{theorem}
We conclude this section by noting that, according to \cite[Table III]{GV2002UMI}
in non-unimodular case, the list of left-invariant harmonic unit vector fields determining 
harmonic maps is poor, contrary to the list obtained in the table above. 
This argument represented a good motivation for us to study unit magnetic vector fields.

\vspace{5mm}

{\bf Acknowledgements.}
This work was born while M.I.M. visited the University of Tsukuba, Japan within the
mobility grant PN-III-P1-1.1-MC-2019-1311.
The first named author is partially supported by JSPS KAKENHI JP19K03461.
The second named author is partially supported by Romanian Ministry of Research, Innovation and Digitization, 
within Program 1 – Development of the national RD system, Subprogram 1.2 – Institutional Performance – RDI excellence funding projects, Contract no.11PFE/30.12.2021.

{\bf Conflict of Interest}: Nil

\medskip

\bibliographystyle{amsplain}

\end{document}